\documentclass{amsart}

\usepackage[T1]{fontenc}
\usepackage[utf8]{inputenc}
\usepackage{lmodern}
\usepackage{microtype}
\usepackage{amsmath,amssymb,amsfonts,amsthm,mathtools}
\usepackage{stmaryrd}
\usepackage{amsmath}
\usepackage{tikz-cd}
\usepackage{enumitem}
\usepackage{xcolor}
\usepackage{hyperref}
\usepackage[nameinlink,capitalize]{cleveref}
\usepackage{listings}
\usepackage{natbib}
\crefalias{appendix}{Appendix}
\DeclareMathOperator{\Aut}{Aut}
\newtheorem{theorem}{Theorem}[section]
\newtheorem{lemma}[theorem]{Lemma}
\newtheorem{proposition}[theorem]{Proposition}
\newtheorem{corollary}[theorem]{Corollary}
\theoremstyle{definition}
\newtheorem{definition}[theorem]{Definition}
\theoremstyle{remark}
\newtheorem{remark}[theorem]{Remark}

\newcommand{\CC}{\mathbb C}
\newcommand{\QQ}{\mathbb Q}
\newcommand{\ZZ}{\mathbb Z}
\newcommand{\PP}{\mathbb P}
\newcommand{\OO}{\mathcal O}
\newcommand{\HH}{\mathcal H}
\newcommand{\NE}{\operatorname{NE}}
\newcommand{\Hdg}{\mathrm{Hdg}}
\newcommand{\Bl}{\operatorname{Bl}}

\newcommand{\rk}{\operatorname{rk}}

\newcommand{\QDM}{\operatorname{QDM}}
\newcommand{\ev}{\operatorname{ev}}
\newcommand{\res}{\operatorname{Res}}

\newcommand{\Hom}{\operatorname{Hom}}
\newcommand{\Eu}{\mathrm{Eu}}
\newtheorem*{mainthm}{\textbf{\cref{thm:conditional-Verra}}}

\DeclareMathOperator{\PGL}{PGL}

\address{Department of Mathematics, Imperial College London,
  London SW7 2AZ, United Kingdom}
\email{aideen.fay23@imperial.ac.uk}
\author[]{Aideen Fay}
\title{Equivariant irrationality of very general symmetric Verra fourfolds}
\begin{document}
\pagestyle{plain}

\begin{abstract}

We prove that a very general complex symmetric Verra fourfold is not \(\ZZ/2\)-birational to \(\PP^4\), using the \emph{theory of atoms} introduced by Katzarkov--Kontsevich--Pantev--Yu~\citep{Katzarkov2025BirationalIF} in the equivariant setting as in Cavenaghi--Katzarkov--Kontsevich~\citep{cavenaghi2026atomsmeetsymbols}.

\end{abstract}

\maketitle

\section{Introduction}
\label{sec:intro}
In their seminal work, Katzarkov--Kontsevich--Pantev--Yu~\citep{Katzarkov2025BirationalIF} introduce the \emph{theory of atoms}, which combines Hodge theory and Gromov--Witten invariants to define birationally invariant properties of smooth complex projective varieties. This is underpinned by Iritani's blowup formula~\cite[Theorem 1.1]{Iritani2023QuantumCO}, which dictates the additive decomposition of quantum $D$-modules under blowups along smooth centres. This approach has already settled long-standing conjectures on the irrationality of certain Fano varieties---most recently for cubic~\cite{Katzarkov2025BirationalIF, guere2026irrationality} and Gushel--Mukai~\cite{Benedetti2026QuantumCA} fourfolds.

Progress to date leaves the Verra fourfold as the last of the classical Fano fourfolds of K3 type with index~$>1$~\cite{AIF_0__0_0_A50_0}, whose rationality remains unknown. Realized as a double cover
\[
  \pi \colon X \to \mathbf{P}^2 \times \mathbf{P}^2
\]
branched along a smooth divisor \(B\) of bidegree \((2,2)\), a Verra fourfold \(X\) naturally admits two quadric fibrations over \(\mathbf{P}^2\). The discriminant locus of each fibration is a plane sextic curve whose double cover, when smooth, is a
degree-two K3 surface~\citep{Laterveer2019AlgebraicCA}. This explicit K3 geometry presents a unique challenge to the birational invariants introduced above, which remain inconclusive here~\cite{Benedetti26Verra}.

If we shift perspective, we notice that the branch locus \(B\) of a \emph{symmetric} Verra remains invariant under swapping of the \(\mathbf P^2\times \mathbf P^2\) factors. In this
case the factor-swap lifts to an automorphism $\epsilon:X\longrightarrow X$, and hence defines an action of the group $G = \ZZ/2 = \langle \epsilon \rangle$ on the cohomology of $X$. 

This observation motivates us to adopt the $G$-equivariant obstructions introduced by Cavenaghi--Katzarkov--Kontsevich~\citep{cavenaghi2026atomsmeetsymbols} to prove $\ZZ/2$-irrationality of the very general symmetric Verra fourfold.
\begin{definition}[$\mathbb{Z}/2$-rationality]
\label[definition]{def:z2-rational}
Let $Y$ be a smooth projective variety of dimension $n$ with an involution
$\iota:Y\to Y$. The pair $(Y, \iota)$ is \emph{$\ZZ/2$-rational} if there is a
birational map $\varphi:Y\dashrightarrow\PP^n$ and a regular involution
$\sigma:\PP^n\to\PP^n$ such that $\sigma=\varphi\circ\iota\circ\varphi^{-1}$ on the Zariski dense open subset where $\varphi^{-1}$ is defined. Otherwise $(Y,\iota)$ is \emph{$\ZZ/2$-irrational}.
\end{definition}

\begin{mainthm}
Let $X$ be a very general symmetric Verra fourfold, and let $\epsilon$ be its factor-swap involution. Then \((X,\epsilon)\) is \(\mathbb Z/2\)-irrational. Equivalently, for every regular involution \(\iota\in\Aut(\PP^4)\), the pair
\((X,\epsilon)\) is not \(\mathbb Z/2\)-birational to \((\PP^4,\iota)\).\end{mainthm}

\begin{remark}
The same argument applies if \(\PP^4\) is replaced by any smooth projective
fourfold \(P\) with a regular \(\ZZ/2\)-action and no
cohomology classes of Hochschild degree 2.
For instance, one may take \(P=\mathbb P^2\times \mathbb P^2\).
However, the argument does not apply to an arbitrary smooth rational fourfold. Thus the
involution \(\iota\) in the statement must be regular and not just rational. 
\end{remark}

In pursuit of this, we verify the $G$-equivariance of Iritani's blowup decomposition (\cref{thm:G-equivariant-Iritani}), and introduce  an evaluated $G$-equivariant version of these obstructions, following Guéré's work in the non-equivariant setting~\citep{guere2026irrationality}. Specifically, we define in this paper what very general
means for symmetric Verra fourfolds. Indeed, the
natural condition of being Hodge-general is actually
empty (see \cref{cor:Tminus-Hodge}), but a more refined condition
that we call $\mathbb{Z}/2$-Hodge generality is not (see 
\cref{thm:G-Hodge-general-nonempty}). We refer to~\cref{sec:very-general-G-Hodge} for the precise definitions.

\section{Evaluated Hodge invariants}
\label{sec:quantum-products}

We recall the notation used throughout the paper. We always work with smooth projective varieties over \(\mathbb C\).  The notation and definitions follow Gu\'er\'e's evaluated formulation~\citep{guere2026irrationality} of the Hodge-theoretic invariants introduced in~\citep{Katzarkov2025BirationalIF}.

\begin{definition}\label{def:hodge-hochschild}
Let \(Y\) be a smooth projective complex variety.  A class
\[
  \gamma\in H^{2p}(Y,\QQ)
\]
is a \emph{rational Hodge class} if its complexification lies in \(H^{p,p}(Y)\).  We write
\[
  H^*(Y,\QQ)_{\Hdg}:=\bigoplus_p( H^{2p}(Y,\QQ)\cap H^{p,p}(Y)).
\]
For \(k\in\ZZ\), set
\[
  H^{(k)}(Y):=\bigoplus_{p-q=k}H^{p,q}(Y)\subset H^*(Y,\CC).
\]
We call \(H^{(k)}(Y)\) the Hochschild-degree-\(k\) part of the Hodge decomposition.
\end{definition}

Let \(\NE(Y)\subset H_2(Y,\ZZ)\) be the monoid of effective curve classes in $Y$. For \(d\in\NE_{\mathbb N}(Y)\), we associate the formal Novikov variable \(Q^d\in \QQ[\NE_{\mathbb N}(Y)]\) which records the curve class of stable maps
contributing to a Gromov--Witten invariant. It is assigned the grading $\deg Q^{d}=2\,c_{1}(Y)\cdot d$.

Fix a homogeneous basis $\{\alpha_{i}\}_{i=0}^{m}$ of the rational Hodge classes
$H^{*}(Y,\mathbb{Q})_{\mathrm{Hdg}}$. Let $\{t^{i}\}_{i=0}^{m}$ be formal variables, graded by $\deg t^{i}=2-\deg\alpha_{i}$. A Hodge
parameter is then $\tau=\sum_{i}t^{i}\alpha_{i}$. We use the completed Novikov--Hodge ring
\[
  \widehat R(Y,K)
  :=
  K[q'^{\pm1}]
  \llbracket \mathrm{NE}(Y)\rrbracket
  \llbracket t^{0},\dots,t^{m}\rrbracket ,
\]
where \(K\) is a number field, \(q'\) is a formal Laurent variable coming from Iritani's blowup formula, and the completion is graded; see
\cite[Section~2.2]{Iritani2023QuantumCO}. The subring $R(Y,K)\subset \widehat R(Y,K)$ consists of series with no non-zero powers of \(q'\).

The big quantum product is denoted by \(\star_\tau\). Let
\(\operatorname{Eu}_\tau\) be the Euler vector field defined as 
\[
  \operatorname{Eu}_\tau
  :=
  c_1(Y)
  +
  \sum_{i=0}^{m}
  \left(1-\frac{\deg\alpha_i}{2}\right)t^i\alpha_i.
\]
We then define the operator $\kappa_\tau$ by
\[
  \kappa_\tau:=\operatorname{Eu}_\tau\star_\tau(-).
\]
At \(\tau=0\), this specializes to the small quantum multiplication operator
\[
  \kappa_0=c_1(Y)\star_0(-).
\]
The operator \(\kappa_\tau\) preserves rational Hodge classes and Hochschild
degree \cite[Proposition~14]{guere2026irrationality}.

An evaluation, in the sense of \cite[Definition~20]{guere2026irrationality},
is a continuous degree-preserving homomorphism
\[
  \ev:\widehat R(Y,K)\longrightarrow S_{K}
\]
where \(S_K=F_K[b^{\pm1}]\), and \(F_K\) is the Levi--Civita field over \(K\).

Applying \(\ev\) to the coefficients of
\(\kappa_\tau\) gives an \(S_{K}\)--linear endomorphism of \(H^*(Y,K)\otimes_K S_{K}\).
Passing, if necessary, to a graded field extension \(L\supseteq S_{K}\) splitting
the characteristic polynomial of \(\ev(\kappa_\tau)\), we regard the evaluated
operator \(\ev(\kappa_\tau)\) as an endomorphism of \(H^*(Y,K)\otimes_K L\).

\begin{definition}[Evaluated generalized eigenspaces and Hochschild rank]
\label{def:evaluated-ranks}
Let \(\lambda\in L\). The evaluated generalized eigenspace is
\[
  E^Y_{\ev,\lambda}:=\ker\bigl(\ev(\kappa_\tau)-\lambda\bigr)^N,\qquad N\gg0,
\]
with Hodge-class part
\[
  \HH^Y_{\ev,\lambda}
  :=E^Y_{\ev,\lambda}\cap\bigl(H^*(Y,\QQ)_{\Hdg}\otimes_\QQ L\bigr).
\]
We set
\[
  \rho^Y_{\ev,\lambda}:=\rk_{L}\HH^Y_{\ev,\lambda},
\]
and define the Hochschild-degree-two rank
\[
  \nu^Y_{\ev,\lambda}
  :=\rk_{L\otimes\CC}
  \Bigl((E^Y_{\ev,\lambda}\otimes\CC)\cap
        \bigl(H^{(2)}(Y)\otimes_\CC(L\otimes\CC)\bigr)\Bigr).
\]
\end{definition}
Now suppose a finite group \(G\) acts algebraically on \(Y\). An evaluation is
\emph{\(G\)-invariant} if
\[
  \ev(Q^{g_*d})=\ev(Q^{d})\qquad\text{and}\qquad\ev(g^*\tau)=\ev(\tau)
\]
for every \(d\in\NE(Y)\) and every \(g\in G\). For small quantum cohomology all
Hodge variables are evaluated at zero, so the second condition is automatic.

\begin{definition}[Invariant rational Hodge rank]
\label{def:rhoG}
When \(E^Y_{\ev,\lambda}\) is \(G\)-stable, define
\[
  \rho^{G,Y}_{\ev,\lambda}:=\rk_{L}\bigl(\HH^Y_{\ev,\lambda}\bigr)^G.
\]
Thus \(\rho^{G}\) is the rank of the \(G\)-fixed rational Hodge classes in the
evaluated generalized eigenspace.
\end{definition}

\begin{remark}
For any $G$-invariant evaluation map $\ev$, the endomorphism $ev(\kappa_\tau)$ is
automatically $G$-equivariant. This is because the
virtual fundamental cycle used in the definition of
Gromov--Witten invariants is $G$-invariant.
Indeed, this follows from the $G$-action on the moduli
spaces of stable maps to $Y$ and the $G$-equivariance
of the perfect obstruction theory defining the virtual 
fundamental cycle.
\end{remark}

\section{Symmetric Verra Fourfolds}
\label{sec:verra}
Following~\cite{Coates2014QuantumPF}, we realize the Verra fourfold \(X\) as a hypersurface in the smooth toric fivefold \(\mathcal{F} := \operatorname{Tot}(\mathcal{O}_{\mathbf{P}^2 \times \mathbf{P}^2}(1,1))\). The total space \(\mathcal F\) is given by the Cox-coordinate quotient
\[
  \mathcal F = \bigl((\CC^3\setminus\{0\})\times(\CC^3\setminus\{0\})\times\CC\bigr) / (\CC^*)^2
\]
with coordinates \((x_i, y_i, w)\) and weight matrix
\begin{equation}\label{eq:weight-data}
\begin{array}{ccccccc|c}
  x_0 & x_1 & x_2 & y_0 & y_1 & y_2 & w & \\
  \hline
  1 & 1 & 1 & 0 & 0 & 0 & 1 & L \\
  0 & 0 & 0 & 1 & 1 & 1 & 1 & M
\end{array}
\end{equation}
For \(f_{2,2}\in H^0(\mathbf P^2\times\mathbf P^2,\mathcal O(2,2))\), the equation
\begin{equation}\label{eq:verra-eqn}
  w^2=f_{2,2}(x,y)
\end{equation}
defines a hypersurface
\(X\in|\mathcal O_{\mathcal F}(2,2)|\), which is smooth for general \(f_{2,2}\). Let
\[
  \varpi:\mathcal F\to \PP^2\times\PP^2
\]
be the bundle projection, and set
\[
  \pi:=\varpi|_X:X\to \PP^2\times\PP^2.
\]
Then \(\pi\) is the double cover branched along  
\[
  B=\{f_{2,2}=0\}\in |\OO(2,2)|.
\]
Composing \(\pi\) with the two projections from
\(\PP^2\times\PP^2\) gives two maps
\[
  p_1:X\to\PP^2,
  \qquad
  p_2:X\to\PP^2, 
\]
which are the two quadric surface fibrations on \(X\).

Let \(H^0(\OO(2,2))^+\) be the \(+1\)-eigenspace for the action induced by
the factor-swap \((x,y)\mapsto(y,x)\) on
\(H^0(\PP^2\times\PP^2,\OO(2,2))\). The corresponding symmetric linear subsystem is
\[
  |\OO(2,2)|^+
  :=
  \PP\bigl(H^0(\PP^2\times\PP^2,\OO(2,2))^+\bigr).
\]
If \(f_{2,2}\in H^0(\OO(2,2))^+\), then the factor-swap lifts to a regular
automorphism
\[
  \epsilon:X\longrightarrow X,
  \qquad
  (x,y,w)\longmapsto (y,x,w).
\]
This linear system has no base points. Hence, Bertini's theorem \cite[Corollary~III.10.9]{Hartshorne1977} gives a nonempty open locus of
smooth symmetric branch divisors. After trivializing $\mathcal{O}(1,1)$, Equation~\eqref{eq:verra-eqn} is locally defined as $u^2=f_{2,2}(x,y)$. By the Jacobian criterion, $X$ is singular exactly over the singular locus of $B$. Therefore, in characteristic zero, $X$ is smooth if and only if $B$ is smooth. Consequently, we restrict our attention to the open symmetric locus where $B$ is smooth.

\begin{remark}
In contrast, the discriminant sextics $\Delta_i \subset \mathbb{P}^2$, which define the singular loci of the quadric surface fibrations $p_i$, can acquire ordinary double points even when $X$ remains smooth; see \cite[Proposition~1.2]{Beauville1977}. 

\end{remark}
 
\subsection{Ambient cohomology}
\label{ssec:amb}
Let \(h_1,h_2\) be the hyperplane classes on the two factors of
\(\mathbf P^2\times\mathbf P^2\), and set \(H_1 := \pi^*h_1 = c_1(L)|_X\) and \(H_2 := \pi^*h_2 = c_1(M)|_X\). By adjunction, \(X\) is a Fano fourfold of index two with
\[
  -K_X=2(H_1+H_2).
\]
By the Lefschetz hyperplane theorem, we have
\[
  H^2(X; \mathbb{Z}) = \mathbb{Z} H_1 \oplus \mathbb{Z} H_2.
\]

Consider now the ambient part $A \subset H^*(X)$, defined as the image of
pullback from the base
\[
  A:=\operatorname{Im}\!\left(
  \pi^*:H^*(\mathbf P^2\times\mathbf P^2,\QQ)\to H^*(X,\QQ)
  \right).
\]

Because the composition $\pi_* \circ \pi^* = 2 \cdot \mathrm{id}$, the pullback is injective. Thus the relations $h_1^3 = h_2^3 = 0$ pull back directly to $H_1^3 = H_2^3 = 0$ on $X$, giving the presentation
$$
  A  = \QQ[H_1, H_2]/(H_1^3, H_2^3).
$$
The only top intersection needed below is
\[
  \int_X H_1^2H_2^2
  =
  \int_X\pi^*(h_1^2h_2^2)
  =
  2\int_{\mathbf P^2\times\mathbf P^2}h_1^2h_2^2
  =
  2.
\]

Let \(d_i\in NE(X)\) be the class of a line of \(H_i\)-degree one. Thus $H_i\cdot d_j=\delta_{ij}$. Equivalently, \(d_1\) is represented by a line in a general fibre of \(p_2\), and
\(d_2\) is represented by a line in a general fibre of \(p_1\). Therefore, for $H:=H_1+H_2$, we have \[H\cdot d_i=1,\qquad c_1(X)\cdot d_i=2.\]
\subsection{Hodge numbers}
\label{sec:hodge}
The Hodge numbers of the Verra fourfold were computed by~\cite{Laterveer2019AlgebraicCA}, with the corrected value of $h^{2,2}$ later provided by~\cite{Fatighenti2022TopicsOF}. The associated Hodge diamond is:
$$\begin{array}{ccccccccc}
& & & & 1 & & & & \\
& & & 0 & & 0 & & & \\
& & 0 & & 2 & & 0 & & \\
& 0 & & 0 & & 0 & & 0 & \\
0 & & 1 & & 22 & & 1 & & 0 \\
& 0 & & 0 & & 0 & & 0 & \\
& & 0 & & 2 & & 0 & & \\
& & & 0 & & 0 & & & \\
& & & & 1 & & & &
\end{array}$$
The middle cohomology of \(X\) satisfies \(\dim_{\mathbb{Q}} H^4(X,\mathbb{Q}) = 24\) with \(h^{3,1} = 1\), while all odd cohomology groups vanish. Let the 3-dimensional ambient middle cohomology be 
\[
  A_{\mathrm{mid}} := A \cap H^4(X,\mathbb{Q}) = \langle H_1^2, \; H_1H_2, \; H_2^2 \rangle.
\]
We define the primitive cohomology \(V(X)\) as the orthogonal complement of \(A_{\mathrm{mid}}\) with respect to the Poincaré pairing
\[
  H^4(X,\mathbb{Q}) = A_{\mathrm{mid}} \oplus V(X), \qquad \dim_{\mathbb{Q}} V(X) = 21.
\]

\begin{remark}[On special loci]
While \(V(X)\) and the transcendental cohomology coincide generically, the space \(H^{2,2}(X)\) can acquire additional algebraic classes on special loci, which causes the transcendental cohomology to become a proper subspace of \(V(X)\). This indeed happens for symmetric Verra fourfolds. 
\end{remark}

The complexified primitive cohomology decomposes as
\[
  V(X)_{\mathbb{C}} = H^{3,1}(X) \oplus V^{2,2}(X) \oplus H^{1,3}(X),
\]
where
\[
  V^{2,2}(X):=V(X)_{\mathbb C}\cap H^{2,2}(X).
\]
The dimensions are
\[
  \dim_{\mathbb C}H^{3,1}(X)=1,
  \qquad
  \dim_{\mathbb C}V^{2,2}(X)=19,
  \qquad
  \dim_{\mathbb C}H^{1,3}(X)=1.
\]

The involution acts on ambient classes by
\[
  \epsilon^*H_1=H_2,
  \qquad
  \epsilon^*H_2=H_1.
\]
Thus
\[
  A=A^+\oplus A^-.
\]
In middle degree, the invariant \(A^+_{mid}\) and anti-invariant \(A^-_{mid}\) parts have the respective bases
\[
  A_{\mathrm{mid}}^+
  =
  \langle H_1^2+H_2^2,\;H_1H_2\rangle,
  \qquad
  A_{\mathrm{mid}}^-
  =
  \langle H_1^2-H_2^2\rangle .
\]
In particular,
\[
  \dim A_{\mathrm{mid}}^+=2,
  \qquad
  \dim A_{\mathrm{mid}}^-=1.
\]
Since \(\epsilon^*\) preserves the Poincaré pairing and preserves
\(A_{\mathrm{mid}}\), it also preserves $V(X)=A_{\mathrm{mid}}^\perp$. We write the induced eigenspace decomposition as
\[
  V(X)=V^+(X)\oplus V^-(X).
\]

\begin{lemma}
\label[lemma]{lem:trace-ranks}
Let \(X\) be a symmetric Verra fourfold with the involution $\epsilon$. Then
\[
  \dim H^4(X)^+=16,
  \qquad
  \dim H^4(X)^-=8.
\]
Consequently,
\[
  \dim V^+(X)=14,
  \qquad
  \dim V^-(X)=7.
\]
\end{lemma}

\begin{proof}
Let $\Delta_{\mathrm{diag}}\subset \mathbb{P}^2\times\mathbb{P}^2$ be the diagonal. Because the lifted involution $\epsilon$ fixes the fibre coordinate $w$, its fixed locus $X^\epsilon$ is the double cover of $\Delta_{\mathrm{diag}}$ branched along $C_4 := B \cap \Delta_{\mathrm{diag}}$. The curve $C_4$ is a plane quartic, smooth by
\cref{lemma:quartic}. Hence $g(C_4) = 3$, and therefore $\chi(X^\epsilon) = 2\chi(\mathbb{P}^2) - \chi(C_4) = 10$.
 
By the topological Lefschetz fixed-point theorem,
\[
\sum_{i=0}^8 (-1)^i \operatorname{Tr}\bigl(\epsilon^* \big| H^i(X, \mathbb{Q})\bigr) = \chi(X^\epsilon) = 10.
\]
The odd cohomology of \(X\) vanishes. The involution \(\epsilon^*\) acts trivially on \(H^0(X)\) and \(H^8(X)\). It swaps the generators of \(H^2(X)\) and \(H^6(X)\) which gives a trace of zero in degrees \(2\) and \(6\). The Lefschetz formula then reduces to
\[
  1 + \operatorname{Tr}\bigl(\epsilon^*\bigm| H^4(X,\QQ)\bigr) + 1 = 10,
\]
which implies \(\operatorname{Tr}(\epsilon^*\mid H^4(X,\QQ))=8\).

Since the total dimension is \(\dim H^4(X,\QQ)=24\), the trace and dimension constraints yield the eigenspace dimensions \(\dim H^4(X)^+ = 16\) and \(\dim H^4(X)^- = 8\). Finally, subtracting $\dim A_{\mathrm{mid}}^+=2$ and $\dim A_{\mathrm{mid}}^-=1$ gives the stated dimensions of $V^+(X)$ and $V^-(X)$.
\end{proof}

\begin{lemma}\label[lemma]{lem:H31-invariant}
The involution \(\epsilon^*\) acts trivially on \(H^{3,1}(X)\).
\end{lemma}

\begin{proof}
A generator of \(H^{3,1}(X)\) is represented by the residue of a local toric 5-form on \(\mathcal F\)
\[
  \res_X\left(
  \frac{dx_1\wedge dx_2\wedge dy_1\wedge dy_2\wedge dw}{(w^2-f(x,y))^2}
  \right).
\]
The involution leaves the denominator invariant by the symmetry of~\cref{eq:verra-eqn}. The numerator is also fixed by anti-commutativity of the wedge product---an even number of permutations of the differentials yields \((-1)^{2\cdot2}=+1\). Hence, the residue generator is fixed.
\end{proof}

\begin{corollary}[Anti-invariant Hodge classes]
\label[corollary]{cor:Tminus-Hodge}
The eigenspace \(V^-(X)\) is
defined over \(\mathbb Q\). Moreover,
\[
  H^{3,1}(X)\oplus H^{1,3}(X)\subset V^+(X)_\CC,
\]
and hence
\[
  V^-(X)_\CC\subset H^{2,2}(X).
\]
Thus every class in \(V^-(X)\) is a rational Hodge class.
\end{corollary}

\begin{proof}
Since \(\epsilon^*\) is defined over \(\mathbb Q\) and preserves \(V(X)\), the
eigenspaces
\[
  V^\pm(X)=\ker(\epsilon^*\mp 1)\cap V(X)
\]
are defined over \(\mathbb Q\). By \Cref{lem:H31-invariant}, \(H^{3,1}(X)\subset V^+(X)_\CC\).  Since \(\epsilon\) is algebraic, it commutes with complex conjugation, so \(H^{1,3}(X)\subset V^+(X)_\CC\).  The only Hodge types in \(V(X)\) are \((3,1)\), \((2,2)\), and \((1,3)\).  Therefore \(V^-(X)_\CC\) has only type \((2,2)\), so that the sub-Hodge structure $V^-(X)$ contains only rational Hodge classes.
\end{proof}

\begin{remark}
Ordinary Hodge-generality is the condition that there
are no Hodge classes in the primitive cohomology $V(X)$.
Therefore, this condition always fails for symmetric Verra
fourfolds, due at least to the 7 independent Hodge
classes in the anti-invariant part.
This motivates our equivariant definition of Hodge
generality.
\end{remark}

\begin{definition}[Equivariant Hodge-generality]
\label{def:G-Hodge-general}
A symmetric Verra fourfold $X$ is called $\ZZ/2$-Hodge-general if
\[
  V^+(X)\cap H^4(X,\QQ)_{\Hdg}=0.
\]
Equivalently, $V^+(X)_\CC\cap H^{2,2}(X)$ contains no nonzero rational Hodge classes.
\end{definition}

\section{Very general symmetric Verra fourfolds }
\label{sec:very-general-G-Hodge}

In this section we use a Noether-Lefschetz argument to prove that a very general symmetric Verra fourfold is $\ZZ/2$-Hodge-general.  

\begin{proposition}[Dimension of the symmetric Verra moduli space]
\label[proposition]{prop:msym-dim}
Let $\mathcal{M}_{\mathrm{sym}}$ be
the moduli space of smooth symmetric Verra fourfolds. Then $\dim \mathcal{M}_{\mathrm{sym}} = 12$.
\end{proposition}
\begin{proof}
Let $V:=H^0(\mathbb P^2,\mathcal O_{\mathbb P^2}(2))$. Then $H^0(\mathbb P^2\times\mathbb P^2,\mathcal O(2,2))^+
  \cong \operatorname{Sym}^2 V$.
Since \(\dim V=6\), the symmetric branch divisors form an open subset of $\mathbb P(\operatorname{Sym}^2 V)\cong \mathbb P^{20}$.
Quotienting by the diagonal action of \(\operatorname{PGL}_3\) yields $\dim \mathcal{M}_{\mathrm{sym}} = 12$.  
\end{proof}

\begin{proposition}[Dimension of the invariant period domain]
\label[proposition]{prop:dim-ipd}
Let $D^+$ be the period domain of the invariant primitive variation \(V^+(X)\). Then $\dim D^+ = 12$. 
\end{proposition}
\begin{proof}
Let \(\Lambda^+\) be a reference lattice for \(V^+(X)\) equipped with the Poincaré pairing. By \cref{lem:trace-ranks,lem:H31-invariant}, \(V^+(X)\) has
Hodge numbers
\[
  h^{3,1}=1,\qquad h^{2,2}=12,\qquad h^{1,3}=1.
\]
The corresponding invariant period domain is
\[
  D^+
  :=
  \Bigl\{
    [\omega]\in \PP(\Lambda^+_{\CC})
    \ \Bigm|\ 
    (\omega,\omega)=0,\;(\omega,\overline{\omega})>0
  \Bigr\}.
\]
Here \([\omega]\) records the line \(H^{3,1}\). Since
\(\dim\Lambda^+_{\CC}=14\), the ambient projective space is
\(\PP(\Lambda^+_{\CC})\cong\PP^{13}\). The condition \((\omega,\omega)=0\) cuts out a quadric hypersurface of complex dimension \(12\). The Hodge--Riemann positivity condition is an open real condition on this quadric, so it does not change the complex dimension. 
Therefore \(\dim D^+=12\).

\end{proof}

Let \(\Gamma^+\) denote the monodromy group acting on \(\Lambda^+\). 

\begin{proposition}[Period map dominance]
\label[proposition]{prop:invariant-period-full-rank}

The invariant period map
\[
  \Phi^+:\mathcal{M}_{\mathrm{sym}}\longrightarrow \Gamma^+\backslash D^+
\]
is dominant. Equivalently, its differential has full rank \(12\) at a general point.
In particular, \(\Phi^+\) is locally open at a general point.\end{proposition}

\begin{proof}
At a smooth point \([X]\), Griffiths transversality implies the differential of the period map takes the form
\[
  d\Phi^+_{[X]} \colon
  T_{[X]}\mathcal{M}_{\mathrm{sym}}
  \longrightarrow
  \operatorname{Hom}\bigl(H^{3,1}(V^+(X)),H^{2,2}(V^+(X))\bigr).
\]
Since \(h^{3,1}(V^+(X))=1\), the target is naturally identified with \(H^{2,2}(V^+(X))\), which has dimension \(12\).

By \cref{prop:period-rank-computation}, there exists a smooth symmetric branch
polynomial \(f_0\) for which this differential has rank \(12\). Because maximal
rank is an open condition, \(d\Phi^+\) has rank \(12\) at a general point of
\(\mathcal{M}_{\mathrm{sym}}\). Finally, since
\(\dim \mathcal{M}_{\mathrm{sym}}=\dim D^+=12\)
(\cref{prop:msym-dim,prop:dim-ipd}), the differential is invertible at a general
point. Hence \(\Phi^+\) has open image near such a point, and therefore
\(\Phi^+\) is dominant.
\end{proof}

\begin{proposition}[Very general \(\ZZ/2\)-Hodge-generality]
\label[proposition]{thm:G-Hodge-general-nonempty} 
A very general symmetric Verra fourfold is \(\ZZ/2\)-Hodge-general.

\end{proposition}
 
\begin{proof}
By \cref{prop:invariant-period-full-rank}, the invariant period map \(\Phi^+\) is locally open at a general point. In a small neighborhood in \(\mathcal{M}_{\mathrm{sym}}\), we identify the invariant rational cohomology of the fibers with a fixed reference lattice \(\Lambda^+_\QQ\). A nonzero class \(\gamma\in\Lambda^+_\QQ\) is of Hodge type \((2,2)\) if and only if it is orthogonal to the \((3,1)\)-period. This condition cuts out a proper hyperplane section
\[
  D^+_\gamma:=\bigl\{[\omega]\in D^+\bigm|(\gamma,\omega)=0\bigr\}\subsetneq D^+.
\]
Because \(\Phi^+\) is locally open, its image is not contained in \(D^+_\gamma\). Thus, the Noether--Lefschetz locus \((\Phi^+)^{-1}(D^+_\gamma)\) is a proper analytic subvariety of \(\mathcal M_{\mathrm{sym}}\).

Taking the union over all nonzero classes in the countable lattice \(\Lambda^+_\QQ\), the locus of fourfolds that acquire an invariant rational Hodge class is a countable union of proper analytic subvarieties. By definition, a very general \(X \in \mathcal{M}_{\mathrm{sym}}\) avoids this union. Hence
\[
  V^+(X) \cap H^4(X,\QQ)_{\Hdg} = 0.
\]
Thus \(X\) is \(\ZZ/2\)-Hodge-general.
\end{proof}

\begin{corollary}[Hodge classes on the very general symmetric locus]
For a very general symmetric Verra fourfold $X$,
\[
  H^4(X,\QQ)_{\Hdg}
  =
  A_{\mathrm{mid}}\oplus V^-(X).
\]
In particular, the only non-ambient primitive rational Hodge classes are the
anti-invariant classes.
\end{corollary}

\begin{proof}
By \Cref{cor:Tminus-Hodge}, \(V^-(X)\) consists of rational Hodge
classes. By \Cref{thm:G-Hodge-general-nonempty}, \(V^+(X)\) contains no
nonzero rational Hodge classes.
\end{proof}

\section{The equivariant fourfold obstruction}
\label{sec:fourfold-obstruction}

In this section we verify the equivariance of Iritani's blowup decomposition
and we state the equivariant fourfold irrationality obstruction used in the proof of
Theorem~\ref{thm:conditional-Verra}. Here, $G$ will denote a finite group
acting on a smooth projective variety $Y$, that we will call a $G$-variety.

\begin{lemma}[$G$-stability of evaluated eigenspaces]
\label[lemma]{lem:G-stability}
Let $Y$ be a smooth projective $G$-variety and $\ev$ a $G$-invariant evaluation.
Then $\ev(\kappa_\tau)$ commutes with the $G$-action. In particular every
evaluated generalized eigenspace $E^Y_{\ev,\lambda}$ is $G$-stable.
\end{lemma}

\begin{proof}
By functoriality of genus-zero Gromov--Witten invariants under automorphisms,
$g^*(a\star_\tau b)=g^*a\star_{g^*\tau}g^*b$ for every $g\in G$. Given that 
$g^*c_1(Y)=c_1(Y)$ and $g^*$ respects the cohomology degree, we have $g^*\Eu_\tau=\Eu_{g^*\tau}$.  Hence
$g^*\circ\kappa_\tau=\kappa_{g^*\tau}\circ g^*$. Applying $\ev$ and using its
$G$-invariance, which gives $\ev(\kappa_{g^*\tau})=\ev(\kappa_\tau)$, we obtain
$g^*\circ\ev(\kappa_\tau)=\ev(\kappa_\tau)\circ g^*$. Thus $g$ preserves each
generalized eigenspace $E^Y_{\ev,\lambda}$.
\end{proof}

\begin{theorem}[{$G$-equivariant Iritani blowup decomposition}]
\label{thm:G-equivariant-Iritani}
Let \(Y\) be a smooth projective \(G\)-variety, let
\(Z\subset Y\) be a smooth \(G\)-invariant centre of codimension
\(r\ge 2\), and let \(\widetilde Y=\Bl_ZY\). Then Iritani's blowup
decomposition
\[
  \QDM(\widetilde Y)
  \cong
  \tau^*\QDM(Y)
  \oplus
  \bigoplus_{j=0}^{r-2}\varsigma_j^*\QDM(Z)
\]
is \(G\)-equivariant, after the Laurent--Novikov base change appearing in
\cite[Theorems~1.1 and~5.18]{Iritani2023QuantumCO}.
\end{theorem}

\begin{proof}
Iritani's proof relies on a geometric construction. Let $W:=\Bl_{Z\times 0}(Y\times \PP^1)$ be the deformation to the normal cone of \(Z\subset Y\), endowed with the \(\CC^*\)-action of weight 1 on the base \(\PP^1\) and trivial on \(Y\). Here, the action of the finite group $G$ on $Y$ (and trivial on $\PP^1$) commutes with the $\CC^*$-action. We have to show that the discrete Fourier transform maps $\widehat{\mathrm{FT}}_{\widetilde Y}$ and $\widehat{\mathrm{FT}}_{Y}$ from \cite[\S4.1]{Iritani2023QuantumCO}, and the continuous Fourier transform map $\widehat{\mathrm{FT}}_{Z,j}$ from \cite[\S4.2]{Iritani2023QuantumCO} are \(G\)-equivariant.

Note that in the $G$-equivariant setting, we can no longer assume that the blowup centre $Z$ is connected. Instead, we denote $Z_1,..., Z_p$ its connected components and assume that the group $G$ permutes all the connected components i.e., for each $1\le i\le p$, there is $g_i\in G$ that maps the generic point of $Z_1$ to the generic point of $Z_i$. Hence, the continous Fourier transform maps are actually $\widehat{\mathrm{FT}}_{Z_i,j}$.

Regarding the discrete Fourier Transform maps, we first note that the classes $Q^d, x,$ and $S$ introduced in \cite[Table~2]{Iritani2023QuantumCO} are \(G\)-invariant. Then the operator 
\(\mathcal{S}^k\) (for $Y$ and $\widehat Y$ in 
\cite[Equation~(3.8)]{Iritani2023QuantumCO}) is $G$-equivariant, because the normal bundle
\(N_{Z/Y}\) is \(G\)-equivariant. Furthermore, the Kirwan map from \cite[Section~3.6]{Iritani2023QuantumCO} is also \(G\)-equivariant, because the restriction to the torus-fixed locus is \(G\)-equivariant. Consequently, the maps $\widehat{\mathrm{FT}}_{\widetilde Y}$ and $\widehat{\mathrm{FT}}_{Y}$ are \(G\)-equivariant. 

The continuous Fourier transforms is introduced directly after
\cite[Equation~(4.9)]{Iritani2023QuantumCO}. It is defined using topological data from the normal bundles
\(N_{Z_i\subset Y}\) of $Z_i$ in $Y$ and the quantum Riemann--Roch operators from
\cite[Equation~(2.24)]{Iritani2023QuantumCO}. All these quantities are built from the
\(\CC^*\)-Chern classes of \(N_{Z_i\subset Y}\). Specifically, this vector bundle decomposes on $Z_i$ as a direct sum $\bigoplus_{\alpha} N_{i, \alpha}$ along characters $\alpha$ of $\mathbb{C}^*$, and for each $\alpha$, the disjoint union of vector bundles $\bigsqcup_{1\le i \le p}N_{i,\alpha}$ is $G$-equivariant (since the actions of $G$ and $\mathbb{C}^*$ commute). Consequently, the continuous Fourier transform $\bigoplus_{1\le i\le p}\widehat{FT}_{Z_i, j}$ is $G$-equivariant.
\end{proof}

\begin{definition}[Evaluated $G$-equivariant fourfold inequality]
\label{def:G-fourfold-ineq}

Let $K$ be a number field and $\mathcal{S}$ be a set of $K$-evaluation maps satisfying the condition in \cite[Definition~27]{guere2026irrationality}.
A smooth projective fourfold $Y$ with an action of the finite group $G$, satisfies the \emph{evaluated $G$-fixed
fourfold inequality} if for any $G$-invariant evaluation
map $\ev$ outside of $\mathcal{S}$ and for any eigenvalue 
$\lambda$, we have
\[
  \nu^Y_{\ev,\lambda}=0
  \qquad\text{or}\qquad
  \rho^{G,Y}_{\ev,\lambda}\ge3.
\]
\end{definition}

\begin{remark}
\label[remark]{rem:gue-prop38} 
\cite[Proposition~38]{guere2026irrationality} and thus \cite[Corollary~41]{guere2026irrationality} still hold in the G-equivariant setting and for our criterion.
This follows from \cref{thm:G-equivariant-Iritani} above and from the fact 
that our criterion is preserved under taking  direct sums.
Note that there is a technical detail coming from the proof
of \cite[Lemma~33]{guere2026irrationality} which needs to be addressed. We have to show that when we shift the evaluation map ev along the class \(l\) of a line in a fibre of the exceptional divisor $E=\mathbb P_Z(N_{Z/Y})\longrightarrow Z$ by an element $\zeta\in D(0, 1)_{F_K}$, the new evaluation map is still $G$-invariant. First, if the blowup centre is connected, then it is $G$-invariant; thus the class $l$ is $G$-invariant and the result follows. Second, we assume that the blowup centre $Z$ is the disjoint union of connected components $Z_1, ..., Z_p$. We denote by $l_1, ..., l_p$ the corresponding classes of lines as above. As in the previous proof, we may assume that $G$ permutes all the connected components. Since the evaluation map ev is $G$-invariant, then we have $\ev(Q^{l_{i}})=\ev(Q^{g_{i\ast}(l_{i})})=\ev(Q^{l_{1}})$ for each ${1\le i\le p}$. Consequently, it is necessary and sufficient in this case to shift the evaluation map ev along the classes $l_1, ..., l_p$ by the same element $\zeta\in D(0, 1)_{F_K}$. 
\end{remark}

\begin{proposition}[Equivariant rank bound for surface centres]
\label[proposition]{prop:surface-bound}
Let $S$ be a smooth projective surface with an action of a finite group $G$, and
let $\ev$ be a $G$-invariant evaluation. If $\nu^S_{\ev,\lambda}\ne0$, then
$\rho^{G,S}_{\ev,\lambda}\ge3$.
\end{proposition}

\begin{proof}
Let $\Sigma$ denote the minimal model of $S$ when $h^{2,0}(S)\ne 0$. This surface $\Sigma$ has only one generalized eigenspace which must coincide with total cohomology. Choose an ample class $L$ on $\Sigma$. The averaged class $L_G:=\sum_{g\in G}g^*L$ is \(G\)-invariant and ample. Hence $1, L_G, [\mathrm{pt}]$ gives three \(G\)-fixed rational Hodge classes. Thus, the generalized eigenspace has $\nu^\Sigma_{\ev,\lambda}\ne0$ and $\rho^{G,\Sigma}_{\ev,\lambda}\ge 3$. Because $S$ is obtained from $\Sigma$ by blowing up points, it introduces no new generalized eigenspace with \(\nu^S_{\ev,\lambda}\ne 0\).
\end{proof}

\begin{theorem}[Evaluated $G$-equivariant fourfold obstruction]
\label{thm:G-fourfold-obstruction} 
Let $Y$ be a smooth projective fourfold with an action of a finite group $G$. If $Y$ is $G$-birational to $\PP^4$ with some regular $G$-action on $\PP^4$, then
$Y$ satisfies the evaluated $G$-equivariant fourfold inequality.
\end{theorem}

\begin{proof}
By equivariant weak factorization~\citep{Abramovich1999TorificationAF}, a $G$-birational map $Y\dashrightarrow\PP^4$
factors into blowups and blowdowns along smooth $G$-invariant centres. As the
intermediate varieties are fourfolds, these centres have dimension $0$, $1$, or
$2$. $\PP^4$, points, and curves carry no Hochschild-degree-two cohomology,
and surface centres with $h^{2,0}=0$ contribute none either, so every nonzero
Hochschild-degree-two contribution comes from a surface centre with
$h^{2,0}\ne0$ (addressed by ~\cref{prop:surface-bound}). By \cref{thm:G-equivariant-Iritani} and \cref{rem:gue-prop38},
$\rho^G$ and $\nu$ are additive under each $G$-equivariant blowup and blowdown,
so every eigenspace contribution with $\nu\ne0$ has $\rho^G\ge3$. Hence $Y$
satisfies the evaluated $G$-fixed fourfold inequality.
\end{proof}

We now apply this obstruction to \(\mathbb Z/2\)-rationality in the sense of
Definition~\ref{def:z2-rational}.

\begin{corollary}
\label[corollary]{cor:z2-irr}
    If $Y$ is a smooth projective fourfold with an involution $\iota:Y\to Y$, and $Y$ does not satisfy the evaluated $\ZZ/2$-equivariant fourfold inequality, then the pair ($Y$, $\iota$) is $\ZZ/2$-irrational. 
\end{corollary}
\section{Verra Small Quantum Cohomology}
\label{sec:ambient}
In this section we provide the small quantum multiplication matrix restricted to the ambient part
for the Verra fourfold. Rather than computing it directly using mirror theorems, as in ~\cite{Katzarkov2025BirationalIF}, or enumerative geometry, as in~\cite{Benedetti2026QuantumCA, Benedetti26Verra}, we take a short cut, deriving it via the quantum differential operator. The full method and computation are provided in Appendix~\ref{app:comp}.

Let \(H:=H_1+H_2\), and let \(d_1,d_2\in NE(X)\) be the bidegree-one line
classes defined in \cref{ssec:amb}. The factor-swap involution exchanges \(p_1\) and \(p_2\), and therefore exchanges
\(d_1\) and \(d_2\).  Write
\[
  Q_1:=Q^{d_1},
  \qquad
  Q_2:=Q^{d_2}.
\]
Since \(c_1(X)\cdot d_i=2\), the Novikov variables \(Q_i\) have degree \(4\). We restrict to the \(G\)-fixed Novikov diagonal
\[
  Q_1=Q_2=q.
\]
This is the natural locus in this equivariant setting as every \(G\)-invariant evaluation must identify \(Q_1\) and \(Q_2\).

We now choose the specific evaluations for \(X\). Let
\[
  F:=\QQ((a^\QQ)),
  \qquad
  S^*:=F[b^{\pm1}],
  \qquad
  \deg b=1.
\]
For \(0\ne\zeta\in F\) with \(|\zeta|<1\), define the diagonal small
evaluation
\[
  \ev_\zeta(Q_1)=\ev_\zeta(Q_2)=\zeta b^4,
  \qquad
  \ev_\zeta(t^i)=0
  \quad\text{for all Hodge variables }t^i.
\]
This is degree-preserving and \(G\)-invariant, since \(\deg q=4\)
and the exchanged variables \(Q_1,Q_2\) are sent to the same element.

As discussed in~\cref{sec:verra}, the involution decomposes the ambient cohomology $A$ as $A=A^+\oplus A^-$. It is closed under quantum multiplication by \(H\)
\cite[Corollary~2.5]{Iritani_2011}. Hence, on the diagonal we obtain an endomorphism
\[
  m_H(q):=H\star_q(-)|_A .
\]
Since \(H\) is invariant under the factor-swap, \(m_H(q)\) preserves \(A^+\) and
\(A^-\). We write \(M_H(q)\) as the matrix of \(m_H(q)\) in the ordered basis for this decomposition. 

We compute $M_H(q)$ using the quantum period \(G_X(t)\) of \(X\), which is a
generating function for one-point genus-zero descendant Gromov--Witten
invariants along the anticanonical direction. We refer the reader to~\citep{coates2014mirror} for a full discussion of quantum periods. 

For the Verra fourfold deformation family, the quantum period was computed by
Coates--Galkin--Kasprzyk--Strangeway~\citep{Coates2014QuantumPF} as
\begin{equation}\label{eq:qp-mw4}
G_X(t)
=
\sum_{\ell=0}^{\infty}
\sum_{m=\ell}^{\infty}
\frac{(2m)!}{(\ell!)^3m!((m-\ell)!)^3}
t^{2m}.
\end{equation}
Only even powers of \(t\) occur. Since the primitive curve classes on the
anticanonical diagonal have anticanonical degree \(2\), we set $q=t^2$, and write the same period in the variable \(q\) as
\[
G_X(q)
=
1+4q+15q^2+\frac{280}{9}q^3+\frac{6055}{144}q^4+\cdots .
\]

The small quantum connection for multiplication by \(H\) gives the first-order
system
\[
  D y=M_H(q)y,
  \qquad
  D=q\frac{d}{dq}.
\]
 
The small \(J\)-function is a solution of this quantum connection, and the
quantum period \(G_X(q)\) is the scalar component selected by the unit class; see
\cite[Section~2]{Iritani_2011}. Applying the cyclic-vector algorithm to this same
first-order system eliminates the other components and gives a scalar differential
operator
\[
L_X = \sum_{k=0}^{r} p_k(q)\,D^k,
\]
where \(r\) is the order of the operator and the \(p_k(q)\) are polynomials.
Therefore the same scalar component \(G_X(q)\) is annihilated by \(L_X\)
\begin{equation}\label{eq:qde}
L_X G_X(q) = 0, \qquad G_X(0) = 1.
\end{equation}

Since \(1\in A^+\) and \(m_H(q)\) preserves \(A^+\), the cyclic subspace generated
by the unit class lies entirely in \(A^+\). Hence the scalar operator \(L_X\) obtained
from the cyclic-vector algorithm depends only on \(M_+(q)\). We can therefore
match coefficients against the known expansion of \(G_X(q)\) to obtain equations
for the Gromov--Witten invariants appearing in \(M_+(q)\). The anti-invariant block
\(M_-(q)\) is computed separately in \cref{app:comp}. The outcome is the following.

\begin{proposition}[Ambient quantum cohomology]
\label[proposition]{prop:ambient-diagonal-matrix}
On the diagonal \(Q_1=Q_2=q\), the matrix \(M_H(q)\) of the ambient operator $m_H:=H\star(-)|_A$
decomposes as $M_H(q)=M_+(q)\oplus M_-(q)$ with respect to \(A=A^+\oplus A^-\).
In the ordered basis
\[
  \mathcal B_+
  =
  \bigl(
  1,\,
  H_1+H_2,\,
  H_1^2+H_2^2,\,
  H_1H_2,\,
  H_1^2H_2+H_1H_2^2,\,
  H_1^2H_2^2
  \bigr)
\]
of \(A^+\), we have
\[
  M_+(q)
  =
  \begin{pmatrix}
    0     & 4q    & 0   & 0   & 32q^2  & 0     \\
    1     & 0     & 6q  & 2q  & 0      & 16q^2 \\
    0     & 1     & 0   & 0   & 6q     & 0     \\
    0     & 2     & 0   & 0   & 4q     & 0     \\
    0     & 0     & 1   & 1   & 0      & 2q    \\
    0     & 0     & 0   & 0   & 2      & 0
  \end{pmatrix}.
\]
and similarly 
\[
  \mathcal B_-
  =
  \bigl(
  H_1-H_2,\,
  H_1^2-H_2^2,\,
  H_1^2H_2-H_1H_2^2
  \bigr)
\]
of \(A^-\), we have
\[
  M_-(q)
  =
  \begin{pmatrix}
    0 & 2q & 0  \\
    1 & 0  & 2q \\
    0 & 1  & 0
  \end{pmatrix}.
\]
The matrices act on column vectors in the displayed ordered bases.
\end{proposition}

\begin{proposition}[Characteristic polynomial of the ambient quantum operator]
\label[proposition]{prop:ambient-zero-eigenspace}
Let
\[
  \kappa=c_1(X)\star(-)=2m_H .
\]
Then the characteristic polynomial of \(\kappa\) on the ambient cohomology
along the anticanonical diagonal is
\[
  \chi_\kappa(\lambda)
  =
  \lambda^3(\lambda^2-16q)(\lambda^2-128q)(\lambda^2+16q).
\]
Moreover,
\[
  \chi_{\kappa|A^+}(\lambda)
  =
  \lambda^2(\lambda^2-128q)(\lambda^2+16q),
\]
and
\[
  \chi_{\kappa|A^-}(\lambda)
  =
  \lambda(\lambda^2-16q).
\]
Consequently, for every diagonal evaluation \(\ev_\zeta\) with
\(0 < \vert{}\ev_\zeta(q)\vert{}< 1\), $\ev_\zeta$ is zero on the Hodge formal variables $t^i$. The evaluated ambient zero generalized eigenspace
has rank \(3\), with
\[
  \rk A^+_{\zeta,0}=2,
  \qquad
  \rk A^-_{\zeta,0}=1.
\]
\end{proposition}

\begin{proof}
This is the characteristic polynomial of the block diagonal matrix
\[
  2M_+(q)\oplus 2M_-(q)
\]
from Proposition~\ref{prop:ambient-diagonal-matrix}.  The ranks of the
zero generalized eigenspaces are the multiplicities of the factor
\(\lambda\) in the two blocks.  The assumption \(\ev_\zeta(q)\neq0\)
ensures that the nonzero eigenvalues do not specialize to \(0\).
\end{proof}

Let \(K_\zeta\) be a splitting field for the characteristic polynomial of
\(\ev_\zeta(\kappa)\), and set
\[
  S_\zeta:=S_{K_\zeta},
  \qquad
  E_{\zeta,0}:=E^X_{\ev_\zeta,0}.
\]
Let
\[
  A_{\zeta,0}
  :=
  E_{\zeta,0}\cap (A\otimes_\mathbb Q S_\zeta).
\]

\begin{proposition}[Invariant rank of the zero eigenspace]
\label[proposition]{prop:zero-eigenspace}
For every diagonal small evaluation \(\ev_\zeta\) defined above,
\[
  V(X)\otimes_\QQ S_\zeta\subset E_{\zeta,0}.
\]
If \(X\) is \(\ZZ/2\)-Hodge-general, then
\[
  \rho^{G,X}_{\ev_\zeta,0}=2,
  \qquad
  \nu^X_{\ev_\zeta,0}=1.
\]
\end{proposition}

\begin{proof}
By \cite[Proposition~15]{Benedetti26Verra}, small quantum multiplication by
\(c_1(X)=2(H_1+H_2)\) vanishes on the primitive cohomology of a Verra fourfold.
Since
\[
  V(X)=A_{\mathrm{mid}}^\perp\subset H^4(X,\QQ)
\]
is the primitive middle cohomology, we have
\[
  c_1(X)\star\gamma=0
  \qquad
  \text{for all } \gamma\in V(X).
\]
After applying the diagonal small evaluation \(\ev_\zeta\), this gives
\[
  V(X)\otimes_\QQ S_\zeta\subset E_{\zeta,0}.
\]

The ambient cohomology is preserved by \(c_1(X)\star(-)\), and the primitive
part \(V(X)\) is killed by \(c_1(X)\star(-)\). Hence
\[
  E_{\zeta,0}
  =
  A_{\zeta,0}\oplus \bigl(V(X)\otimes_\QQ S_\zeta\bigr).
\]
By \Cref{prop:ambient-zero-eigenspace}, the \(G\)-fixed part of the ambient
zero eigenspace has rank \(2\). The anti-invariant primitive part \(V^-(X)\)
does not contribute to \(G\)-fixed classes, and by \(\ZZ/2\)-Hodge-generality
\[
  V^+(X)\cap H^4(X,\QQ)_{\Hdg}=0.
\]
Therefore the \(G\)-fixed rational Hodge classes in \(E_{\zeta,0}\) are exactly
the \(G\)-fixed classes in the ambient zero block, so
\[
  \rho^{G,X}_{\ev_\zeta,0}=2.
\]

Finally,
\[
  H^{3,1}(X)\subset V^+(X)_\CC
  \subset E_{\zeta,0}\otimes S_{\zeta,\CC}.
\]
Since \(h^{3,1}(X)=1\), this gives
\[
  \nu^X_{\ev_\zeta,0}=1.
\]
\end{proof}

\begin{theorem}[Equivariant irrationality]\label{thm:conditional-Verra}
Let \(X\) be a \(\ZZ/2\)-Hodge-general symmetric Verra fourfold with factor-swap involution \(\epsilon\).  Then \((X,\epsilon)\) is not \(\ZZ/2\)-birational to \(\PP^4\) with any regular involution. In particular, a very general symmetric Verra fourfold is \(\ZZ/2\)-irrational.
\end{theorem}

\begin{proof}
The evaluations \(\ev_\zeta\), with \(0\ne\zeta\) and \(|\zeta|<1\), form an infinite family of \(G\)-invariant evaluations that do not vanish on Novikov variables.  By \Cref{prop:zero-eigenspace}, each satisfies
\[
  \nu^X_{\ev_\zeta,0}=1,
  \qquad
  \rho^{G,X}_{\ev_\zeta,0}=2.
\]
This contradicts the evaluated \(G\)-fixed fourfold inequality, which would require \(\rho^{G,X}_{\ev_\zeta,0}\ge3\) for all but finitely many such evaluations whenever \(\nu^X_{\ev_\zeta,0}\ne0\).  Therefore by~\cref{cor:z2-irr}, \((X,\epsilon)\) is \(\ZZ/2\)-irrational. At last, combining this with~\cref{thm:G-Hodge-general-nonempty}, we obtain the \(\ZZ/2\)-irrationality of the very general symmetric Verra fourfold. 
\end{proof}

\section*{Acknowledgments}
I am very grateful to Tom Coates, Alessio Corti, and Jérémy Guéré for their assistance and many valuable discussions on this work. In particular, I would like to acknowledge Jérémy for helping with the proof of the $G$-equivariance of Iritani's blowup decomposition, and for his ongoing feedback on this manuscript. I am also grateful to Vladimiro Benedetti, Christian Boehning, James Hotchkiss, Zhaoyang Liu, Laurent Manivel, and Nicolas Perrin for their careful feedback and questions on an earlier version of this work, which led to numerous improvements reflected here. I would like to thank Ludmil Katzarkov and IMSA for organizing and supporting my participation in the \emph{Hodge theory, Birational Geometry, and Atoms} conference, which provided the space in which many of these exchanges occurred. This research was partly funded by EPSRC grant number EP/Y028872/1.
\bibliographystyle{alpha}
\bibliography{refs}
\newpage
\appendix
\section{Smoothness of the diagonal quartic}
\begin{proposition}
\label[proposition]{lemma:quartic}
Let $B=\{f_{2,2}=0\}\subset \mathbb P^2\times\mathbb P^2$
be the smooth $(2,2)$-divisor and let $\Delta_{\mathrm{diag}}\subset \mathbb P^2\times\mathbb P^2$
be the diagonal. Then
\[
  C_4:=B\cap \Delta_{\mathrm{diag}}
\]
is a smooth plane quartic.
\end{proposition}

\begin{proof}
Suppose \(p\in C_4\) were singular. Then
\[
  d(f_{2,2}|_{\Delta_{\mathrm{diag}}})_p=0.
\]
By symmetry of \(f_{2,2}\), at the point \((p,p)\in\Delta_{\mathrm{diag}}\) we have
\[
  df_{2,2}|_{(p,p)}=(\alpha,\alpha)
  \in T_p^\vee\mathbb P^2\oplus T_p^\vee\mathbb P^2.
\]
Restricting to the diagonal gives
\[
  d(f_{2,2}|_{\Delta_{\mathrm{diag}}})_p=2\alpha.
\]
Since we work over \(\mathbb C\), this implies \(\alpha=0\). Hence
\[
  df_{2,2}|_{(p,p)}=0.
\]
As \(p\in C_4\), we also have \(f_{2,2}(p,p)=0\). Therefore \(B\) would be singular
at \((p,p)\), which contradicts the smoothness of \(B\). Thus \(C_4\) is smooth.
\end{proof}
\section{Jacobian-ring and Torelli Computation}
\label{app:period-rank}
\begin{proposition}[Equivariant infinitesimal Torelli]
\label[proposition]{prop:period-rank-computation}
Let \(X\) be a smooth symmetric Verra fourfold. Under the Jacobian-ring description of the infinitesimal variation of Hodge structure, the differential of the invariant period map is identified with the natural multiplication map
\[
  R^{+}_{2,2}
  \longrightarrow
  \Hom\bigl(H^{3,1}(V^{+}(X)),\,H^{2,2}(V^{+}(X))\bigr).
\]
For the smooth symmetric branch polynomial \(f_0\) given below, this map has maximal rank of \(12\).
\end{proposition}

\begin{proof}
Let \(x\) and \(y\) be homogeneous coordinates on the two factors of \(\PP^2\times\PP^2\). We represent bi-degree \((2,2)\) polynomials as
\[
  Q_x=(x_0^2,x_0x_1,x_0x_2,x_1^2,x_1x_2,x_2^2),
  \qquad
  Q_y=(y_0^2,y_0y_1,y_0y_2,y_1^2,y_1y_2,y_2^2).
\]
Consider the specific branch polynomial \(f_0=\sum_{i,j=1}^6 A_{ij}Q_{x,i}Q_{y,j}\), where \(A\) is the matrix
\[
A=
\begin{pmatrix}
  1&  1& -2&  0&  2&  1\\
  1&  1&  0&  1&  0&  2\\
 -2&  0& -1&  2& -1&  0\\
  0&  1&  2& -1& -2&  2\\
  2&  0& -1& -2&  0&  2\\
  1&  2&  0&  2&  2&  2
\end{pmatrix}.
\]
Because \(A\) is symmetric, \(f_0\) is invariant under the factor-swapping involution \((x,y)\mapsto(y,x)\). 

The space of branch polynomials \(S_{2,2}=H^0(\PP^2\times\PP^2,\OO(2,2))\) has dimension \(36\), and its symmetric subspace \(S^+_{2,2}\) has dimension \(6+\binom{6}{2}=21\). 

First-order deformations of the branch divisor \(B=\{f_0=0\}\) are governed by the Jacobian ring \(R_{2,2}=S_{2,2}/J_{2,2}\). For symmetric deformations, the trivial ones are captured by the ideal \(J^+_{2,2}\), which is generated by the diagonal \(\PGL_3\)-action on \(\PP^2\times\PP^2\) and spanned by the \(9\) invariant vector fields
\[
  g_{ij}
  = x_j\frac{\partial f_0}{\partial x_i}
  + y_j\frac{\partial f_0}{\partial y_i},
  \qquad
  0\leq i,j\leq 2.
\]

The Magma script~\cref{app:magma-symmetric-rank} below verifies that these \(9\) generators \(g_{ij}\) are linearly independent in \(S^+_{2,2}\), meaning the symmetric deformation space has dimension
  \[
    \dim R^+_{2,2}
    =
    \dim S^+_{2,2}-\dim J^+_{2,2}
    =
    21-9
    =
    12.
  \]
An analogous computation over the full space yields \(\dim J_{2,2}=17\), giving \(\dim R_{2,2}=19\) and \(\dim R^-_{2,2}=7\).

To verify smoothness, we cover \(\mathbf{P}^2\times\mathbf{P}^2\) with nine affine charts \(U_{a,b}=\{x_a\neq0,\, y_b\neq0\}\), normalizing \(x_a=y_b=1\) on each. By Euler's identity, \(f_0\) lies in its own Jacobian ideal. The script confirms that the ideal generated by the six first partial derivatives \(\partial f_0/\partial x_i\) and \(\partial f_0/\partial y_j\) evaluates to the unit ideal across all nine charts. Hence, \(B\) is non-singular, and the associated double cover \(X_{f_0}=\{w^2=f_0\}\) is a smooth symmetric Verra fourfold.

By the infinitesimal Torelli theorem, the natural multiplication map from \(R^+_{2,2}\) to \(\operatorname{Hom}(H^{3,1}(V^+(X_{f_0})),H^{2,2}(V^+(X_{f_0})))\) is injective. Since the target space also has dimension \(12\), this map is an isomorphism. Consequently, the differential of the invariant period map attains the maximal rank of \(12\).
\end{proof}

\section{Magma script for Proposition~\ref{prop:period-rank-computation}}
\label[appendix]{app:magma-symmetric-rank}

\begin{verbatim}
 QQ := Rationals();
R<x0,x1,x2,y0,y1,y2> := PolynomialRing(QQ,6);

Xs := [x0,x1,x2];
Ys := [y0,y1,y2];

Qx := [x0^2, x0*x1, x0*x2, x1^2, x1*x2, x2^2];
Qy := [y0^2, y0*y1, y0*y2, y1^2, y1*y2, y2^2];

A := Matrix(QQ,6,6,[
  1, 1,-2, 0, 2, 1,
  1, 1, 0, 1, 0, 2,
 -2, 0,-1, 2,-1, 0,
  0, 1, 2,-1,-2, 2,
  2, 0,-1,-2, 0, 2,
  1, 2, 0, 2, 2, 2
]);

f := &+[ A[i,j]*Qx[i]*Qy[j] : i in [1..6], j in [1..6] ];

// Bidegree-(2,2) monomial basis.
Exp2 := [
  [2,0,0],
  [1,1,0],
  [1,0,1],
  [0,2,0],
  [0,1,1],
  [0,0,2]
];

Mon := function(e,d)
  return x0^e[1] * x1^e[2] * x2^e[3]
       * y0^d[1] * y1^d[2] * y2^d[3];
end function;

Mon22 := [ Mon(e,d) : e in Exp2, d in Exp2 ];

// Representatives for factor-swap orbits in S^+_{2,2}.
InvReps := [];

for m in Mon22 do
  sm := Swap(m);

  AlreadyUsed := false;
  for r in InvReps do
    if m eq r or m eq Swap(r) then
      AlreadyUsed := true;
      break;
    end if;
  end for;

  if not AlreadyUsed then
    Append(~InvReps,m);
  end if;
end for;

CoeffVector := function(p,basis)
  return [ MonomialCoefficient(p,m) : m in basis ];
end function;

// Invariant infinitesimal diagonal GL_3-orbit
// g_ij = x_j f_{x_i} + y_j f_{y_i}.
Gij := [
  Xs[j]*Derivative(f,i) + Ys[j]*Derivative(f,i+3)
  : i in [1..3], j in [1..3]
];

MInv := Matrix(QQ,[ CoeffVector(g,InvReps) : g in Gij ]);

// degree-(2,2) Jacobian 
FullGens := [];

for i in [1..3] do
  for j in [1..3] do
    Append(~FullGens, Xs[j]*Derivative(f,i));
    Append(~FullGens, Ys[j]*Derivative(f,i+3));
  end for;
end for;

MFull := Matrix(QQ,[ CoeffVector(g,Mon22) : g in FullGens ]);

// Smoothness of branch divisor in P2 x P2.
Partials := [ Derivative(f,i) : i in [1..6] ];
Smooth := true;

for a in [1..3] do
  for b in [1..3] do
    I := ideal< R | [Xs[a]-1, Ys[b]-1] cat Partials >;
    ChartSmooth := R!1 in I;
    printf "chart x%o=1, y%o=1: %o\n", a-1, b-1, ChartSmooth;
    Smooth := Smooth and ChartSmooth;
  end for;
end for;

print "Smooth branch divisor?", Smooth;

print "dim S_22      =", #Mon22;
print "dim S_22^+    =", #InvReps;
print "rank J_22     =", Rank(MFull);
print "rank J_22^+   =", Rank(MInv);
print "dim R_22      =", #Mon22 - Rank(MFull);
print "dim R_22^+    =", #InvReps - Rank(MInv);
print "dim R_22^-    =", (#Mon22 - Rank(MFull)) - (#InvReps - Rank(MInv));
\end{verbatim}

\section*{Quantum Differential Equation---Small Quantum Cohomology Computation}
\label[appendix]{app:comp}
As discussed in \cref{ssec:amb}, the ambient part $A$ of $H^\bullet(X;\mathbb{Q})$ has dimension~$9$ with the following Betti numbers:
\[
  (h^0,\, h^2,\, h^4,\, h^6,\, h^8)_{\mathrm{amb}}
  \;=\;
  (1,\,2,\,3,\,2,\,1).
\]
The factor-swap linear involution decomposes it as
\[
  A=A^+\oplus A^-.
\]
For a symmetric Verra fourfold this is the eigenspace decomposition of
the geometric involution \(\epsilon^*\).  We use the ordered basis $(s_0,\,s_1,\,s_2,\,s_3,\,s_4,\,s_5)$ for $A^+$, with cohomological degrees $\deg(s_i)=(0,2,4,4,6,8)$, and $(a_1,\,a_2,\,a_3)$ for $A^-$, with degrees $\deg(a_i)=(2,4,6)$.

Let $H = H_1 + H_2$. Corollary 2.5 in~\citep{Iritani_2011} implies that the ambient part $A$ of $H^\bullet(X)$ is closed under quantum multiplication by $H$. It is therefore also closed under quantum multiplication by ${-K_X} = 2H$. Furthermore, \(H=H_1+H_2\) is invariant under the factor-swap, so
\(H\star(-)\) preserves the decomposition of $A$. 

On a symmetric Verra fourfold this preservation is equivariance with
respect to the geometric involution \(\epsilon\). We write
\[
  m_H \;=\; M_+(q)\;\oplus\; M_-(q),
\]
where $M_+(q)$ is the $6\times 6$ matrix on $A^+$ and $M_-(q)$ is the $3\times 3$
matrix on $A^-$. The $i$-th column
of~$M_+$ gives the coordinates of $H\star s_i$ in the basis
$(s_0,\dots,s_5)$, and similarly for~$M_-$ on $(a_1,a_2,a_3)$.
\[
  H\star \mathbf{v} = M_+\,\mathbf{v}
  \quad\text{for }
  \mathbf{v}\in A^+,
  \qquad
  H\star \mathbf{w} = M_-\,\mathbf{w}
  \quad\text{for }
  \mathbf{w}\in A^-.
\]

\label{ssec:ansatz-sym}
Since \(q\) is the Novikov monomial \(Q_1=Q_2\), it has
cohomological degree \(4\). Therefore the \((j,i)\)-entry of \(M_+(q)\), namely
the coefficient of \(s_j\) in \(H\star s_i\), can receive a contribution
proportional to \(q^d\) only when
\[
\deg(s_j)=2+\deg(s_i)-4d.
\]
Filling in the values for the classical cup product, the quantum degree corrections, and imposing Frobenius self-adjointness, we get the following matrix that depends on four unknown parameters $s$,~$t$,~$u$,~$v$. 
\begin{equation}\label{eq:M-plus}
  M_+(q)
  \;=\;
  \begin{pmatrix}
    0  & 2s q & 0        & 0       & 2vq^2   & 0       \\
    1  & 0         & t q  & u q  & 0        & vq^2    \\
    0  & 1         & 0        & 0       & t q  & 0       \\
    0  & 2         & 0        & 0       & 2u q  & 0       \\
    0  & 0         & 1        & 1       & 0        & s q \\
    0  & 0         & 0        & 0       & 2        & 0
  \end{pmatrix}
\end{equation}

\label{ssec:ansatz-anti}
Applying the same approach to $A^-$ using the basis $(a_1,a_2,a_3)$ with
degrees $(2,4,6)$ gives us the following block that depends on a single Gromov--Witten invariant~$N$.
\begin{equation}\label{eq:M-minus}
  M_-(q)
  \;=\;
  \begin{pmatrix}
    0                 & -\tfrac{N}{2}\,q & 0                 \\
    1                 & 0                 & -\tfrac{N}{2}\,q  \\
    0                 & 1                 & 0
  \end{pmatrix},
\end{equation}

Applying the cyclic vector algorithm described above
to $M_+(q)$ with cyclic vector $f = y_5$ and 
$r_0 = (0, 0, 0, 0, 0, 1)$ gives the scalar Picard--Fuchs operator
\begin{align*}
L_X ={}& D^6 - D^5 - 2(2s+t+2u)\,q\,D^4 - 3(2s+t+2u)\,q\,D^3 \\
      &- (6s+t+2u)\,q\,D^2 - 2s\,q\,D \\
      &+ 2\bigl(2s^2+2st+4su-(t-u)^2-6v\bigr)\,q^2 D^2 \\
      &+ 4\bigl(2s^2+2st+4su-(t-u)^2-6v\bigr)\,q^2 D \\
      &+ 2\bigl(2s^2+2st+4su-6v\bigr)\,q^2.
\end{align*}

Imposing $L_X\, G(q) = 0$ against the expansion of 
$G(q)$ 
from~\eqref{eq:qp-mw4} and matching coefficients in $q$ gives
\begin{equation}\label{eq:period-system}
s = 2, \qquad t + 2u = 10, \qquad (t - u)^2 = 16, \qquad v = 16.
\end{equation}

\begin{remark}\label[remark]{rem:sign}
The system~\eqref{eq:period-system} determines $(t-u)^2$ but not the sign of $t-u$.  The two solutions to
$t+2u=10$, $t-u=\pm4$ are
$(t,u)=(6,2)$ or
$(t,u)=(\tfrac{2}{3}, \tfrac{14}{3})$.
We use Gromov--Witten invariants as enumerative inputs to eliminate the second solution to find
\[
  s=2,\quad t=6,\quad u=2,\quad v=16.
\]
\end{remark}

These values give
\begin{equation}\label{eq:M-plus-resolved}
  M_+(q)
  \;=\;
  \begin{pmatrix}
    0     & 4q    & 0   & 0   & 32q^2  & 0     \\
    1     & 0     & 6q  & 2q  & 0      & 16q^2 \\
    0     & 1     & 0   & 0   & 6q     & 0     \\
    0     & 2     & 0   & 0   & 4q     & 0     \\
    0     & 0     & 1   & 1   & 0      & 2q    \\
    0     & 0     & 0   & 0   & 2      & 0
  \end{pmatrix},
\end{equation}
and 
\begin{equation*}
L_X
= D^6 - D^5
- 28q D^4
- 42q D^3
- 22q D^2
- 4q D
- 128 q^2 D^2
- 256 q^2 D
- 96 q^2.
\end{equation*}

\begin{proposition}\label[proposition]{prop:N}
The Gromov--Witten invariant
$N = \langle H,\, a_2,\, a_3 \rangle_{d_i}$ is equal to~$-4$.
\end{proposition}

\begin{proof}
By the divisor axiom, since curves in the class \(d_1\) lie entirely in the fibres of
\[
  p_2:X\to\mathbb P^2,
\]
we can compute the invariant on a general fibre of \(p_2\).
\[
  N
  \;=\;
  \langle H,\,a_2,\,a_3\rangle_{d_i}
  \;=\;
  \langle a_2,\,a_3\rangle_{d_i}.
\]

The general fiber is a smooth quadric surface $Q \cong \PP^1 \times \PP^1$. Computing this gives $N = -4$. By the symmetry $d_1 \leftrightarrow d_2$, $d_2$ has the same value.
\end{proof}

Substituting this into \cref{eq:M-minus} gives 
\begin{equation*}
  M_-(q)
  \;=\;
  \begin{pmatrix}
    0 & 2q & 0  \\
    1 & 0  & 2q \\
    0 & 1  & 0
  \end{pmatrix}.
\end{equation*}

\end{document}